\documentclass[a4paper,11pt]{article}
\usepackage{latexsym}
\usepackage{amssymb}
\usepackage{enumerate}
\usepackage{xcolor}
\usepackage{commath}
\usepackage{amsfonts}
\usepackage{amsmath,amsthm}
\usepackage{hyperref}
\usepackage{algorithm}
\usepackage{algorithmic}
\usepackage{framed}
\usepackage{boites}
\usepackage[pdftex]{graphicx}

\newenvironment{@abssec}[1]{%
    \if@twocolumn

      \section*{#1}%
    \else

      \vspace{.05in}\footnotesize
      \parindent .2in
 {\upshape\bfseries #1. }\ignorespaces
    \fi}

    {\if@twocolumn\else\par\vspace{.1in}\fi}

\newenvironment{keywords}{\begin{@abssec}{\keywordsname}}{\end{@abssec}}

\newenvironment{AMS}{\begin{@abssec}{\AMSname}}{\end{@abssec}}

\newcommand\keywordsname{Key words}
\newcommand\AMSname{AMS subject classifications}
\newcommand\AMname{AMS subject classification}
\newcommand\restr[2]{{
\left.\kern-\nulldelimiterspace 
#1 
\vphantom{|} 
\right|_{#2} 
}}
\newtheorem{theorem}{Theorem}[section]
\newtheorem{lemma}[theorem]{Lemma}
\newtheorem{corollary}[theorem]{Corollary}
\newtheorem{proposition}[theorem]{Proposition}
\newtheorem{remark}[theorem]{Remark}

\newtheorem{problem}{Problem}
\newtheorem{prob}[theorem]{Problem}
\newtheorem{conjecture}[theorem]{Conjecture}

\newcommand{\RR}{\mathbb{R}}

\renewcommand{\SS}{\mathbb{S}}

\def\XXint#1#2#3{{\setbox0=\hbox{$#1{#2#3}{\int}$}
\vcenter{\hbox{$#2#3$}}\kern-.5\wd0}}

\newcommand{\link}{\mathop{\circ\kern-.35em -}}

\newcommand{\ol}{\overline}
\newcommand{\pa}{\partial}

\newcommand{\dv}{\mathop{\mathrm{div}}}

\newcommand{\gr}{\nabla}

\newcommand{\al}{\alpha}

\newcommand{\De}{\Delta}

\newcommand{\si}{\sigma}

\newcommand{\om}{\omega}
\newcommand{\Om}{\Omega}
\newcommand{\rn}{{\mathbb{R}}^N}
\newcommand{\sg}{\sigma}

\newcommand\setbld[2]{\left\{ #1 \;\middle |\; #2\right\}}

\newcommand{\C}{\mathcal{C}}
\newcommand{\cC}{\mathcal{C}}

\title{\bf On a two-phase Serrin-type problem \\ and its numerical computation\thanks{This research was partially supported by the Challenging Exploratory Research No.16K13768 of Japan Society for the Promotion of Science and the 	
Grant-in-Aid for JSPS Fellows No.18J11430.}}

\author{Lorenzo Cavallina\thanks{Research Center for Pure and Applied Mathematics,
Graduate School of Information Sciences, Tohoku
University, Sendai, 980-8579, Japan ({\tt cava@ims.is.tohoku.ac.jp}, {\tt  yachimura@ims.is.tohoku.ac.jp}).} \, and Toshiaki Yachimura\footnotemark[2]
}
\date{}

\begin{document}

\maketitle

\begin{abstract}
We consider an overdetermined problem of Serrin-type with respect to an operator in divergence form with piecewise constant coefficients. We give sufficient condition for unique solvability near radially symmetric configurations by means of a perturbation argument relying on shape derivatives and the implicit function theorem. This problem is also treated numerically, by means of a steepest descent algorithm based on a Kohn--Vogelius functional.  
\end{abstract}



\pagestyle{plain}
\thispagestyle{plain}

\section{Introduction and main results}\label{introduction}
Let $(D,\Omega)$ be a pair of sufficiently smooth bounded domains of $\RR^{N}$ ($N\geq2$) such that $\overline{D} \subset \Omega$. Moreover, let $n$ denote the outward unit normal vector of $\Omega$. 
In this paper, for $c\in \RR$, we consider the following overdetermined problem:
\begin{equation}\label{odp}
\begin{cases}
-\dv(\sigma \gr u)=1 \quad \textrm{ in }\Omega,\\
u=0\quad \textrm{ on }\partial\Omega,\\
\partial_n u= c \quad \textrm{ on }\partial\Omega, 
\end{cases}
\end{equation}
where $\sigma = \sigma(x)$ is the piecewise constant function given by 
\begin{equation*}
\sigma(x) = 
\begin{cases}
\si_c \quad &\text{in} \,\, D, \\
1 \quad &\text{in} \,\, \Omega \setminus D  
\end{cases}
\end{equation*}
and $\si_c$ is a positive constant such that $\si_c \neq 1$. 

Notice that such an overdetermined problem does not admit a solution in general, and that, whenever it does, the parameter $c$ must be equal to $c(\Omega)=-|\Omega|/|\pa\Omega|$ by a simple integration by parts. In what follows, we will say that a pair of domains $(D,\Omega)$ is a solution of problem \eqref{odp} whenever problem \eqref{odp} is solvable for $\sg=\sg(D,\Omega)$. 
We now define the so-called inner problem and outer problem associated to problem \eqref{odp}. 
\begin{problem}[Inner problem]
For a given domain $\Omega$ and a real number $0<V_0<|\Omega|$, find a domain $D\subset\ol D \subset \Omega$ with volume $|D|=V_0$, such that the pair $(D,\Omega)$ is a solution of the overdetermined problem \eqref{odp}.
\end{problem}

\begin{problem}[Outer problem]
For a given domain $D$ and a real number $V_0>|D|$, find a domain $\Omega\supset \ol D$ with volume $|\Omega|=V_0$, such that the pair $(D,\Omega)$ is a solution of the overdetermined problem \eqref{odp}. 
\end{problem}

\begin{figure}[h]
\centering
\includegraphics[width=0.45\linewidth]{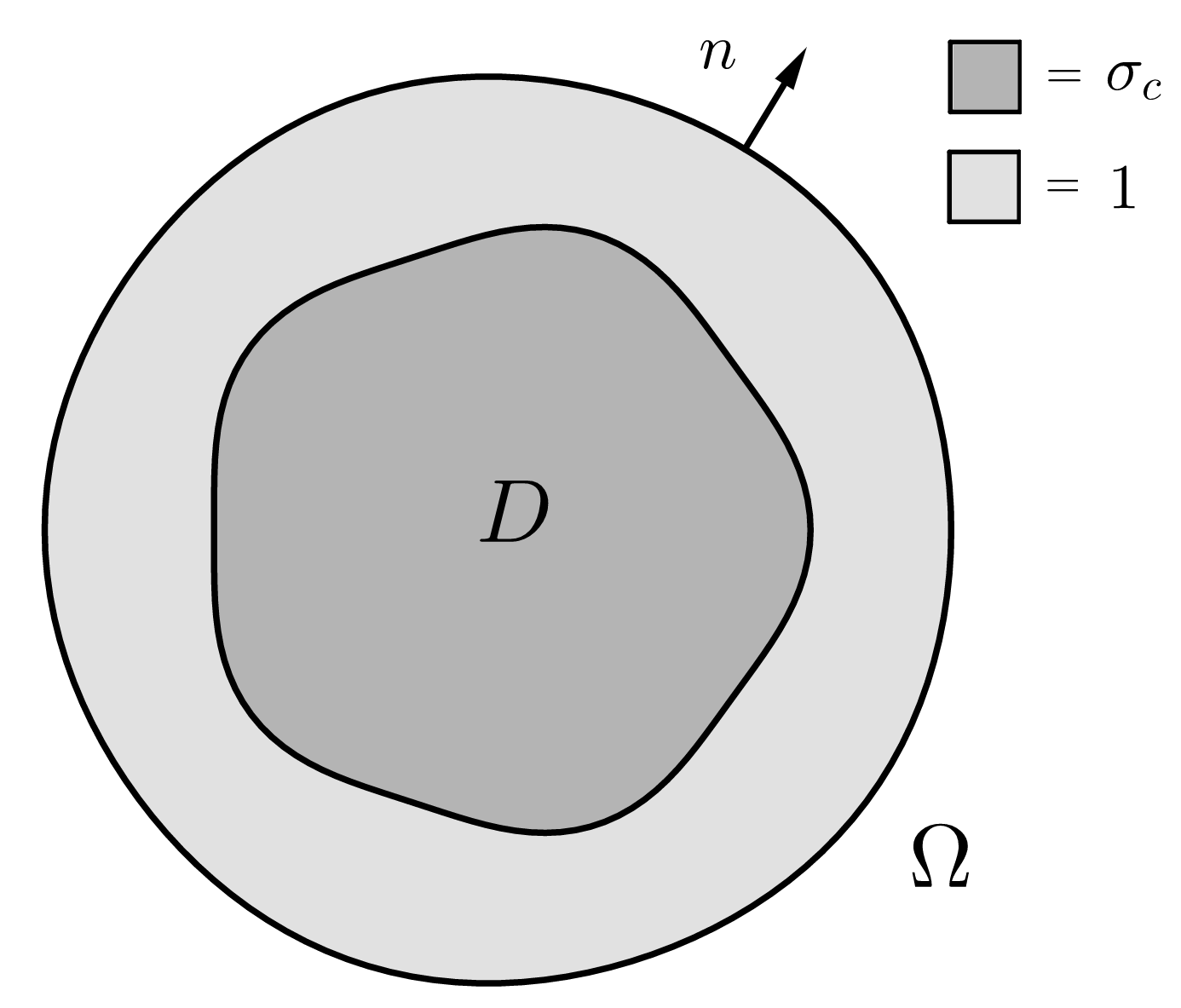}
\caption{Problem setting} 
\label{nonlinear}
\end{figure}

Problem 2 naturally arises as the optimality condition of the following constrained maximization problem for a given $D$: 
\begin{equation}\label{max pb}
    \max_{|\Om|=V_0} \int_\Om \sg |\gr v|^2 = \max_{|\Om|=V_0} \int_\Om v,
    \end{equation}
where the maximization is considered over all domains $\Om$ that satisfy $\Om\supset\ol D$ and $v$ is the solution to the following problem with Dirichlet boundary condition. 
\begin{equation}\label{pb d}
\begin{cases}
-\dv(\sg \gr v)=1 \quad\text{ in } \Om,\\
v=0 \quad \text{ on }\pa\Omega.
\end{cases}
\end{equation}

Physically speaking, problem \eqref{max pb} consists in finding the best coating of a given body $D$ in order to maximize its torsional rigidity under mass constraint. This maximization problem has been analyzed in \cite{cava2018}. In particular, the optimality condition $\pa_n v=c$ on $\pa\Om$, although not explicitly stated there, is a consequence of the computations in \cite[Theorem 3.3]{cava2018}, adapted to the case of volume preserving perturbations acting on a general $\Om$. Finally, we remark that the results of \cite{cava2018} imply that Problem 2 and the constrained maximization problem \eqref{max pb} are not equivalent. Indeed, although all sufficiently smooth solutions of \eqref{max pb} must solve Problem 2, not all solutions to Problem 2 are maximizers of \eqref{max pb}. A counterexample is given in \cite[Theorem 4.5]{cava2018}, where it is shown that the configuration given by two concentric balls is not even a local maximizer for \eqref{max pb} when $\sg_c<1$.

As far as Problem 1 is concerned, by reasoning as above, we see that it can be linked to \eqref{max pb} in the following indirect way.
Problem 1 gives a necessary condition for those domains $D$, which make a given $\Om$ optimal in the sense of \eqref{max pb}.
On a final note, we remark that Problem 1 does not yield the optimality condition for the maximization problem 
\begin{equation}\label{korejanai}
    \max_{|D|=V_0} \int_\Om \sg |\gr v|^2 = \max_{|D|=V_0} \int_\Om v, 
\end{equation}
where the maximization is taken over all domains $D$ satisfying $\ol D\subset \Om$. 
Indeed, the optimality condition of the maximization problem \eqref{korejanai} is given by an overdetermined condition involving the gradient of $v$ on the interface $\pa D$. This follows from the fact that the shape derivative of the Dirichlet energy $\int_\Om \sg |\gr v|^2$ vanishes for all volume preserving perturbations of the domain $D$ (see \cite[Remark 4.4]{cavaphd}). Similar shape optimization problems involving two-phase conductors have been studied in several papers, especially eigenvalue problems. We refer to \cite{CL1996, MT1997A, MT1997B, CMS2009, CLM2012, L2014} and the references given there. 

When $D$ is empty (one-phase setting), the overdetermined problem \eqref{odp} has been studied in the celebrated paper of Serrin \cite{Se1971}. He proved that, in that case, the overdetermined problem \eqref{odp} is solvable if and only if the domain $\Omega$ is a ball. His proof relied on the method of moving planes introduced by Alexandrov \cite{Ale1958}. That is why the overdetermined problem in the one-phase setting is called {\em Serrin's problem}. Many mathematicians, inspired by the work of Serrin, studied similar overdetermined problems for various operators, where the overdetermination consists in prescribing the value of the normal derivative on the boundary (such problems are usually referred to as overdetermined problems of {\em Serrin-type} in the literature). See for example \cite{BNST2008, BHS2014, MagnaniniPoggesi2017, NT2018} and references therein.     

Another geometrical setting for a similar kind of overdetermined problem is obtained when $D$ is a hole, that is, $\Omega\setminus \ol D$ is a doubly connected domain. There are many results concerning overdetermined problems where the value of the normal derivative is prescribed on one of the two connected components of the boundary of $\Omega\setminus \ol D$. Such problems are usually called overdetermined problems of {\em Bernoulli-type}. The main difference with the usual Serrin-type problems is that the solution of a Bernoulli-type problem is not necessarily radially symmetric, even when the overdetermination consists in the normal derivative being constant on the boundary. Indeed, the part of $\pa(\Omega\setminus \ol D)$ where the overdetermination is imposed behaves like a free boundary, that inherits its shape from the geometry of the other component. Theoretical and numerical results about Bernoulli-type problems can be seen in \cite{Beurling 1957, altcaff, FR1997, Henrot Shahgholian 2002, BCT2005, IKP2006, HIKKP2009, LP2012, BPST2017}, and the references therein. 

To our knowledge, there are only a few results concerning problem \eqref{odp} when $D$ is not empty. The paper \cite{camasa} dealt with the inner problem of the overdetermined problem \eqref{odp}. The authors proved the local existence and uniqueness for the inner problem near concentric balls. They also treated the overdetermined problem of two-phase heat conductors and gave symmetry results for the domain.

The purpose of this paper is twofold. First, we study the local existence and uniqueness for the outer problem near concentric balls. Second, we deal with the numerical computation of the solution to the outer problem. 
In what follows, we state the main results of this paper. 

\begin{theorem}\label{mainthm1}
Let us define
\begin{equation*}
\begin{aligned}
s(k)&= \frac{k(N+k-1)-(N+k-2)(k-1)R^{2-N-2k}}{k( N+k-1)+k(k-1)R^{2-N-2k}} \text{ for }k = 1,2,\ldots,\\
\Sigma&=\setbld{s\in (0,\infty)}{s=s(k)
\,\text{ for some }k = 1,2,\ldots}.
\end{aligned}
\end{equation*}
and let $B_{R} \subset B_{1}$ denote concentric balls of radius $R$ and $1$ respectively. 
If $\si_c \notin \Sigma$, then for every domain $D$ of class $\cC^{2,\alpha}$ sufficiently close to $B_{R}$ in the $\cC^{2,\alpha}$-norm (in a sense made more precise in the statement of Theorem \ref{ift applied}), there exists a domain $\Omega$ of class $\cC^{2,\alpha}$ sufficiently close to $B_{1}$ in the $\cC^{2,\alpha}$-norm such that Problem $2$ admits a solution for the pair $(D,\Omega)$. 
\end{theorem}
From Theorem \ref{mainthm1}, Problem $2$ has a solution near concentric balls except for specific values of the coefficients. 
Similar techniques have been used in the context of the control of free boundaries such as the control of the Bernoulli overdetermined problem \cite{KKL2014} and the control of a droplet shape via surface tension \cite{LW2015}.

Moreover, in order to solve the overdetermined problem \eqref{odp} numerically, we consider the Kohn--Vogelius functional introduced by the paper of Kohn and Vogelius \cite{KV1987} from the viewpoint of impedance computed tomography: 
\begin{equation}\label{def kv}
\mathcal{F}(\Omega)=\int_\Omega \sg |\gr v-\gr w|^2, 
\end{equation}
where $v$ is the solution of the Dirichlet problem \eqref{pb d} and $w$ is the solution of the following Neumann problem
\begin{equation}\label{pb n}
\begin{cases}
-\dv(\sg \gr w)=1 \quad\text{ in } \Om,\\
\pa_n w=c \quad \text{ on }\pa\Omega,\\
\int_{\partial\Omega} w= 0.
\end{cases}
\end{equation}
Functionals of the type \eqref{def kv} have been widely used not only in the field of impedance computed tomography but also in free boundary problems, see \cite{EH2012, BBPST2013}. 
Notice that, by definition \eqref{def kv}, the functional $\mathcal{F}$ is always nonnegative. Moreover, when $\mathcal{F}(\Omega)=0$ then $\gr v=\gr w$ in $\Omega$ and thus, by the normalization condition in \eqref{pb n}, $v=w$. In other words, the solutions of the outer problem coincide with the zeros of $\mathcal{F}$.
Therefore, we seek for the zeros of the Kohn--Vogelius functional in order to find the solutions of the outer problem. Notice that Theorem \ref{mainthm1} ensures that this procedure will yield a unique solution if the core $D$ is ``sufficiently close to a ball". 

Since $\mathcal{F}(\Omega)\ge 0$ for all admissible domains $\Omega$, we are going to look for those shapes that minimize $\mathcal{F}$ (and hopefully make it vanish). In other words, we consider the following minimization problem with volume constraint 
\begin{equation}\label{optipro1 intro}
\min_{\abs{\Omega} = V_{0}} \mathcal{F}(\Omega),   
\end{equation}
where the minimum is taken over all admissible domains $\Omega$ such that $\ol D\subset \Omega$.
This task will be performed numerically by a gradient descent algorithm. The steepest descent direction associated to $\mathcal{F}$ will be computed by means of the shape derivative of the Kohn--Vogelius functional $\mathcal{F}$.
\begin{theorem}\label{shape kv intro}
The Kohn--Vogelius type functional $\mathcal{F}$ defined by \eqref{def kv} is shape differentiable at $\Omega$. Moreover, for any smooth $h:\rn\to\rn$ whose support is compactly contained in $\rn\setminus \ol D$, we have  
\begin{equation*}
\mathcal{F}'(\Omega)(h)=\int_{\pa\Omega}\left\lbrace -|\gr w|^2+2(1+cH)w
-|\gr v|^2+2c^2\right\rbrace h\cdot n, 
\end{equation*}
where $H$ is the additive curvature defined by \eqref{mean curvature}. 
\end{theorem}
Combining the result of Theorem \ref{shape kv intro} with the augmented Lagrangian method based on \cite{NW2006, PT2018}, we solve the minimization problem \eqref{optipro1 intro} numerically by a gradient descent method.

This paper is organized as follows. In Section \ref{preonshape}, we give some notations and preliminaries on shape and tangential calculus for shape derivative. In Section \ref{proof of existence and uniqueness}, we prove Theorem \ref{mainthm1} by using shape derivatives and the implicit function theorem for Banach spaces. In Section \ref{shape derivative of KV}, we compute the shape derivative of Kohn--Vogelius functional. In Section \ref{gradient descent method with volume constraint}, we explain the augmented Lagrangian method and our algorithm for the minimization problem \eqref{optipro1 intro}. In Section \ref{Numerical results}, we show the numerical results based on our algorithm introduced in Section \ref{gradient descent method with volume constraint}. In Section \ref{pro and con}, we state some open problems and conjectures.  

\section{Preliminaries on shape and tangential calculus}\label{preonshape}
In this section we will introduce the concept of shape derivatives and related tools. The topic is too old and deep to be treated exhaustively in this paper; thus we refer the interested reader to the monographs \cite{SG, HP2005, SZ1992}.
\subsection{Shape derivatives}
Let us first introduce some basic notation. Let $\om\subset\rn$ be a smooth domain at which we will compute the derivative of a shape functional $J$ (we will, therefore, require $J(\widetilde{\om})$ to be defined at least for all domains $\widetilde{\om}$ ``sufficiently close'' to the reference domain $\om$).
Let $h:\rn\to\rn$ be a smooth vector field. For $t>0$ small enough the perturbation of the identity ${\rm Id}+th:\rn\to\rn$ is a diffeomorphism. Let $\omega_t=({\rm Id}+t h)(\om)$ denote the deformed domain. 
The shape derivative of $J$ at $\om$ with respect to the perturbation field $h$ is then defined as
\begin{equation*}
J'(\om)(h)=\lim_{t\to 0} \frac{J(\om_t)-J(\om)}{t}.
\end{equation*}
Of course, the definition above can be extended to functionals that take several domains as input as well.

The concept of shape derivative can be applied to shape functionals that take values in a general Banach space too. 
A fairly common example is given by a smoothly varying family of sufficiently smooth real-valued functions $f_t$ defined on the set $\om_t=({\rm Id} +t h)(\om)$ (in many practical applications $f_t$ is the solution to some boundary value problem defined on the perturbed domain $\om_t$). Since each $f_t$ lives in a different domain $\om_t$, the shape derivative $f'$ has to be defined in an indirect way (see \cite{SZ1992}), that is
\begin{equation}\label{f' by material deri}
f'=\dot f - \gr f\cdot h,     
\end{equation}
where $\dot f$ is the \emph{material derivative} of $f_t$, defined as
\begin{equation*}
\dot f= \restr{\frac{d}{dt}}{t=0} f_t\circ \left({\rm Id} + th\right).    
\end{equation*}

In what follows we will give the classical Hadamard formulas for computing the derivative of an integral over a domain $\om$, or a surface integral over the boundary $\pa\om$, whose integrand also depends on $\om$. These formulas will be our primary tool in computing shape derivatives in this paper (we refer to \cite[Theorem 5.2.2, p.194 and Proposition 5.4.4, p.215]{HP2005} for the details).

\begin{lemma}[Hadamard formulas]\label{hadamard formula}
Let $\om$ be a smooth domain of $\rn$ with outer unit normal $n$. For a smooth perturbation field $h:\rn\to\rn$, set $\om_t=({\rm Id}+th)(\om)$. For small $t>0$, let $f_t$ and $g_t$ be smooth real valued functions defined on $\om_t$ and $\pa \om_t$ respectively. Suppose that $f_t$ and $g_t$ are shape differentiable in the sense of \eqref{f' by material deri} with shape derivatives $f'$ and $g'$. Set $J_1(t)=\int_{\om_t} f_t$ and $J_2(t)=\int_{\pa\om_t} g_t$. Then $J_1$ and $J_2$ are differentiable at $t=0$ and the following holds:
\begin{equation*}
J_1'(0) = \int_{\om} f' + \int_{\pa\om} f\,h\cdot n, \quad J_2'(0)= \int_{\pa\om} g' + \int_{\pa\om} \left( H g + \pa_n g\right) \,h\cdot n,
\end{equation*}
where $H$ is the additive curvature defined by \eqref{mean curvature}.
\end{lemma}
\subsection{Tangential calculus}
In this subsection we will briefly introduce the basic differential operators from tangential calculus and their fundamental properties. 
In what follows $\omega$ will be a sufficiently smooth domain of $\rn$ and $n$ will denote its outward unit normal. Furthermore, until the end of this subsection the letters $f$ and $g$ will be used to denote $\cC^{2,\al}$ functions defined on $\pa\om$ that take values in $\RR$ and $\rn$ respectively. We define
\begin{eqnarray*}
\gr_\tau f= \gr \tilde{f}- \left(\gr \tilde{f}\cdot n \right)n\quad \text{(tangential gradient)},\\
{\dv}_\tau g = \dv \tilde{g}- \left( \gr \tilde{g}\, n \right)\cdot n \quad \text{(tangential divergence),}
\end{eqnarray*}
where $\tilde{f}$ and $\tilde{g}$ are $\cC^{2,\al}$ extensions of $f$ and $g$ to a neighborhood of $\pa\omega$ and $\gr\tilde{g}$ is the Jacobian matrix of $\tilde{g}$. It is well known that the definitions above do not actually depend on the choice of extensions.
The following tangential version of integration by parts holds true for all smooth $f$, $g$ and $\om$ (see \cite[(5.64), p.221]{HP2005}):
\begin{equation}\label{tangential integration by parts}
\int_{\pa\omega} \gr_\tau f\cdot g = -\int_{\pa\omega} f\, {\dv}_\tau g.
\end{equation}

We introduce the following definition for the additive curvature of $\pa\omega$:
\begin{equation}\label{mean curvature}
H={\dv}_\tau n\quad \text{ on }\pa\omega.
\end{equation}
Notice that the definition above coincides with the sum of the principal curvatures of $\pa\omega$. In particular $H$ is positive whenever $\omega$ is strictly convex and $H\equiv (N-1)/R$ if $\omega$ is a ball of radius $R$. 

Finally, we introduce the tangential analogue of the Laplacian (also known as Laplace--Beltrami operator):
\begin{equation}\label{laplace beltrami op}
\Delta_\tau f = {\dv}_\tau \gr_\tau f\quad\text{ on }\pa\omega.
\end{equation}
Now, we recall the following classical decomposition formula for the Laplace operator that holds for every smooth function $\varphi$ defined on $\ol\omega$ (see \cite[Proposition 5.4.12]{HP2005}):
\begin{equation}\label{decomp lapl}
\Delta \varphi = \pa_{nn}\varphi + H \pa_n \varphi + \Delta_\tau \varphi \quad \text{ on }\pa\om.
\end{equation}

\section{Local existence and uniqueness for the outer problem near concentric balls}\label{proof of existence and uniqueness}

It is clear that the overdetermined problem \eqref{odp} is solvable if $(D,\Omega)$ are concentric balls. Indeed, if $(D,\Omega)$ are concentric balls, then the unique solution of problem \eqref{pb d} will be radial, and thus it will automatically satisfy the two boundary conditions of \eqref{odp}, independently of $\sg_c$. 
The converse is not true. Here, we will show the existence of infinitely many nontrivial solutions $(D,\Omega)$ of the outer problem (Problem 2) by means of a perturbation argument based on the following version of the implicit function theorem for Banach spaces (see \cite[Theorem 2.7.2, pp.34--36]{Ni2001} for a proof). 

\begin{theorem}[Implicit function theorem]\label{ift}
Suppose that $\mathcal{X}$, $\mathcal{Y}$ and $\mathcal{Z}$ are three Banach spaces, $\mathcal{U}$ is an open subset of $\mathcal{X}\times\mathcal{Y}$, $(x_0,y_0)\in \mathcal{U}$, and $\Psi:\mathcal{U}\to \mathcal{Z}$ is a Fr\'echet differentiable mapping such that $\Psi(x_0,y_0)=0$. Assume that the partial derivative $\pa_y \Psi(x_0,y_0)$ with respect to the variable $y$ at $(x_0,y_0)$ is a bounded invertible linear transformation from $\mathcal{Y}$ to $\mathcal{Z}$. 
Then there exists a neighborhood $\mathcal{U}_0$ of $x_0$ in $\mathcal{X}$ and a unique continuous function $g:\mathcal{U}_0\to \mathcal{Y}$ such that $g(x_0)=y_0$, $(x, g(x))\in \mathcal{U}$ and $\Psi(x, g(x))=0$ for all $x\in\mathcal{U}$.
Moreover, the function $g$ is Fr\'echet differentiable in $\mathcal{U}_0$ and its Fr\'echet differential $g'$ can be written as 
\begin{equation}\label{second part}
 g^\prime (x) = - \pa_y \Psi(x, g(x))^{-1}\, \pa_x \Psi(x, g(x)) \quad \textrm{ for } x\in \mathcal{U}_0.
\end{equation}
\end{theorem}
\subsection{Preliminaries}
We now introduce the functional setting for the proof of Theorem \ref{mainthm1}. Let $D$ and $\Om$ be concentric balls of radius $R$ and $1$ respectively ($0<R<1$), whose common center can be thought to be at the origin. 
Fix $\alpha\in (0,1)$ and let $h\in \mathcal{C}^{2,\alpha}(\rn,\rn)$ be a sufficiently small perturbation field such the map ${\rm Id} + h$ is a diffeomorphism from $\rn$ to $\rn$, and such that
\begin{equation}\label{h=fn h=gn}
h= f\, n \quad \textrm{ on }\pa D\quad\textrm{ and } \quad h = g\, n \quad\textrm{ on }\pa\Omega,
\end{equation}
where $f$ and $g$ are given functions of class $\mathcal{C}^{2,\al}$ on $\pa D$ and $\pa \Omega$ respectively.
Next we define the perturbed domains
\begin{equation*}
\Omega_g=({\rm Id}+h)(\Omega) \quad \textrm{ and }\quad D_f = ({\rm Id}+ h) (D). 
\end{equation*}
We will also require $f$ and $g$ to be sufficiently small, so that the inclusion $\ol{D_f}\subset \Om_g$ holds true.
We consider the following Banach spaces (equipped with the standard norms):
\begin{eqnarray}\nonumber
& \mathcal{X}=\setbld{f\in C^{2,\al}(\pa D)}{\int_{\pa D} f\, =0},\quad 
\mathcal{Y}=\setbld{g\in C^{2,\al}(\pa\Om)}{\int_{\pa\Om} g\, =0}, \\
&\mathcal{Z}=\setbld{\psi\in C^{1,\al}(\pa\Om)}{\int_{\pa\Om} \psi\,  =0}.\nonumber
\end{eqnarray}
As done in \cite[Chapter 6]{camasa}, we will apply Theorem \ref{ift} to the mapping $\Psi: \mathcal{X}\times\mathcal{Y}\to\mathcal{Z}$, defined by:
\begin{equation}\label{psi is dare}
\Psi(f,g) = \left\lbrace \pa_{n_g} v_{f,g} - c_g  \right\rbrace J_\tau (g) \quad \textrm{for } (f,g)\in \mathcal{X}\times\mathcal{Y}.
\end{equation}
Here $v_{f,g}$ denotes the solution $v(D_f,\Om_g)$ to the Dirichlet problem \eqref{pb d} corresponding to the deformed configuration $(D_f,\Omega_g)$, similarly 
$n_g$ denotes the outer normal of $\Omega_g$. Moreover, by a slight abuse of notation, the notation $\pa_{n_g} v_{f,g}$ is used to represent the function of value
\begin{equation*}
\gr v_{f,g}\left(x+g(x)\, n(x) \right)\cdot n_g(x+g(x)\,n(x))\quad \textrm{ at any } x\in\pa\Omega.
\end{equation*}
Finally, the constant $c_g$ is just $c(\Omega_g)= -|\Omega_g|/|\pa \Omega_g|$ and the term $J_\tau(g)>0$ is the tangential Jacobian associated to the transformation $x\mapsto x+g(x)\,n(x)$ defined as 
\begin{equation*}
    J_{\tau}(g) = \det\left( {\rm Id} + Dh \right)\norm{({\rm Id} + Dh)^{-T}n}, 
\end{equation*}
where $h$ is the extension of $g$ defined by \eqref{h=fn h=gn} and $\norm{\cdot}$ is the Euclidean norm (see \cite[Definition 5.4.2, p.213]{HP2005}). By definition, we have that $\Psi(f,g)=0$ if and only if the pair $(D_f,\Om_g)$ is a solution of the overdetermined problem \eqref{odp}. In particular we know that $\Psi(0,0)$ vanishes because, since $(D,\Om)$ are concentric balls, $\pa_n v$ is constant on $\pa\Om$.  
Finally, notice that the term $J_\tau(g)>0$ has been added for technical reasons: namely to ensure that $\Psi(f,g)$ has vanishing integral over $\pa\Omega$ (in other words, it belongs to $\mathcal{Z}$) for all $(f,g)\in \mathcal{X}\times\mathcal{Y}$.

\subsection{The derivative of $\Psi$}
The map $\Psi$ is Fr\'echet differentiable in a neighborhood of $(0,0)\in \mathcal{X}\times\mathcal{Y}$. This can be proved in a standard way by following the ideas of \cite[Theorem 5.3.2, p.203]{HP2005} with the help of Schauder's theory for elliptic operators with piecewise constant coefficients. We refer to \cite[Lemma 5.1]{cavaphd} for a complete proof.
As a consequence, for $f\in \mathcal{X}$ and $g\in\mathcal{Y}$, the partial Fr\'echet derivatives $\pa_x \Psi(0,0)(f)$ and $\pa_y \Psi(0,0)(g)$ coincide with the following G\^ateaux derivatives:
\begin{equation*}
\pa_x \Psi(0,0)(f) = \restr{\frac{d}{dt}}{t=0} \Psi(tf,0) \quad\textrm{ and }\quad \pa_y \Psi(0,0)(g) = \restr{\frac{d}{dt}}{t=0} \Psi(0,tg). 
\end{equation*}
Let $v_t=v(D_{tf}, \Omega_{tg})$. In what follows we will employ the use of the following characterization of the shape derivative $v'$ (see \cite[Proposition 3.1]{cava2018}).

\begin{lemma}\label{lemma v' pm}
For every $(f,g)\in \mathcal{X}\times\mathcal{Y}$, the map $t\mapsto v_t$ is shape differentiable, with shape derivative $v'$. Moreover $v'$ can be decomposed as the sum of $v'=v'_-+v'_+$, where $v'_\pm$ are the solutions to the following boundary value problems.

\noindent\begin{minipage}{.5\linewidth}
\begin{equation}\label{u'in}
\begin{cases}\De v'_{-}=0 \quad\mbox{ in } D\cup (\Om\setminus\ol{D}),\\
[\sg\, \pa_n v'_{-}]=0 \quad\mbox{ on }\pa D,\\
[v'_{-}]=-[\pa_n v]f\quad\mbox{ on }\pa D,\\
v'_{-}=0 \quad \mbox{ on }\pa\Om.
\end{cases}
\end{equation}
\end{minipage}%
\begin{minipage}{.5\linewidth}
\begin{equation}\label{u'out}
\begin{cases}
\De v'_{+}=0 \quad \mbox{ in } D\cup (\Om\setminus\ol{D}),\\
[\sg\, \pa_n v'_{+}]=0 \quad\mbox{ on }\pa D,\\
[v'_{+}]=0\quad \mbox{ on }\pa D,\\
v'_{+}= -\pa_n v \,g \quad \mbox{ on }\pa\Om.
\end{cases}
\end{equation}
\end{minipage}
\vspace{4mm}

\noindent In the above, we used square brackets to denote the jump of a function across the interface $\pa D$. Notice that $v'_\pm$ defined above are the shape derivatives of the maps $t\mapsto 
v(D,\Omega_{tg})$ and $t\mapsto v(D_{tf},\Omega)$ respectively.
\end{lemma}

\begin{theorem}\label{Psi'=something}
The Fr\'echet derivative $\Psi'(0,0)$ defines a mapping from $\mathcal{X}\times\mathcal{Y}$ to $\mathcal{Z}$ by the formula
\begin{equation*}
\Psi'(0,0)(f,g) = \pa_n v' + \pa_{nn}v \, g,
\end{equation*}
where $\pa_{nn}v = \left(D^2 v\; n\right) \cdot n$.
In particular, following the notation of Lemma \ref{lemma v' pm}, we have the following expression for the partial Fr\'echet derivatives as well:
\begin{align}
\pa_x \Psi(0,0)(f) &= \pa_n v'_-, \label{deri f} \\
\pa_y \Psi(0,0)(g) &= \pa_n v'_+ + \pa_{nn} v\, g. \label{deri g} 
\end{align}
\end{theorem}
\begin{proof}
Fix $(f,g)\in \mathcal{X}\times\mathcal{Y}$. As before, put $v_t=v(D_{tf},\Omega_{tg})$. Similarly, let $c(t)=c(\Om_{tg})$, $n_t=n_{tg}$ and $J_\tau(t)=J_\tau(tg)$. Since $\Psi$ is Fr\'echet differentiable, we can compute its Fr\'echet derivative as a G\^ateaux derivative as follows:
\begin{equation*}
\Psi'(0,0)(f,g)=\restr{\frac{d}{dt}}{t=0}\Psi(tf,tg)= \restr{\frac{d}{dt}}{t=0} \Big\lbrace \left( \gr v_t\cdot n_t\right)\circ \left( {\rm Id}+tgn \right) - c(t)  \Big\rbrace J_\tau (t).
\end{equation*}
Since $J_\tau (0)\equiv 1$ and $\pa_n v\equiv c(0)$ on $\pa\Omega$, we have
\begin{equation}\label{Psi'=}
\Psi'(0,0)(f,g)= \restr{\frac{d}{dt}}{t=0} \Big\{ \left(\gr v_t\cdot n_t \right)\circ \left({\rm Id}+ tgn\right) - c(t) \Big\}.  
\end{equation}
By a standard calculation with Lemma \ref{hadamard formula} at hand we get
\begin{equation*}
c'(0)= -\restr{\frac{d}{dt}}{t=0} \frac{|\Omega_{tg}|}{|\pa \Omega_{tg}|}= -\frac{1}{|\pa \Omega|} \int_{\pa\Omega}(1+cH) g =0,
\end{equation*}
where in the last equality we used the fact that $H$ is constant on $\pa \Omega$ and that $g$ has vanishing integral by hypothesis.
By Hopf lemma and the boundary condition in \eqref{pb d}, it is clear that $\gr v_t \cdot n_t = -|\gr v_t|<0$ on $\pa\Omega$. Therefore, 
\begin{eqnarray}\nonumber
\Psi'(0,0)(f,g)= -\restr{\frac{d}{dt}}{t=0} |\gr v_t|\circ \left({\rm Id}+ tgn\right) = \\
-\frac{1}{|\gr v|} \Big(\gr v\cdot \gr v' + (D^2 v\, \gr v)\cdot gn  \Big) 
= \pa_n v' +\pa_{nn} v\, g.
\end{eqnarray}
The representation formulas for the partial Fr\'echet derivatives $\pa_x \Psi(0,0)$ and $\pa_y \Psi(0,0)$ follow immediately.
\end{proof}

\subsection{Applying the implicit function theorem}

In order to apply Theorem \ref{ift}, we will need the following explicit representation for the shape derivatives $v'_\pm$ by means of their spherical harmonic expansion. 
Let $\{Y_{k,i}\}_{k,i}$ ($k\in \{0,1,\dots\}$, $i\in\{1,2,\dots, d_k\}$) denote a maximal family of linearly independent solutions to the eigenvalue problem
\begin{equation*}
-\De_{\tau} Y_{k,i}=\lambda_k Y_{k,i} \quad\textrm{ on }\SS^{N-1},
\end{equation*}
with $k$-th eigenvalue $\lambda_k=k(N+k+2)$ of multiplicity $d_k$ and normalized in such a way that $\norm{Y_{k,i}}_{L^2(\SS^{N-1})}=1$. Here $\De_\tau$ stands for the Laplace--Beltrami operator on the unit sphere $\SS^{N-1}$ defined as \eqref{laplace beltrami op}. Such functions, usually referred to as \emph{spherical harmonics} in the literature, form a complete orthonormal system of $L^2(\SS^{N-1})$.
This fundamental property of spherical harmonics turns out to be very useful in computing the solutions of PDE's in radially symmetric domains by applying the method of separation of variables.
We refer to \cite[Proposition 3.2]{cava2018} for a proof of the following result. 
\begin{lemma}\label{lemma v' expansion}
Assume that, for some real coefficients $\alpha_{k,i}^\pm$, the following expansions hold true  for all $\theta\in\SS^{N-1}$:
\begin{equation}\label{h_in h_out exp}
f(R\theta)=\sum_{k=1}^\infty\sum_{i=1}^{d_k}\al_{k,i}^- Y_{k,i}(\theta), \quad 
g(\theta)=\sum_{k=1}^\infty\sum_{i=1}^{d_k}\al_{k,i}^+ Y_{k,i}(\theta).
\end{equation} 
Then, the functions $v'_\pm$ defined in Lemma \ref{lemma v' pm} admit the following explicit expression for $\theta\in\mathbb{S}^{N-1}$:
\begin{equation}\label{u'pm li seme}
v'_\pm(r\theta)=\begin{cases}\displaystyle
\sum_{k=1}^\infty\sum_{i=1}^{d_k}\al_{k,i}^\pm B_k^\pm r^k Y_{k,i}(\theta) \quad &\text{ for } r\in[0,R],\\
\displaystyle
\sum_{k=1}^\infty\sum_{i=1}^{d_k}\al_{k,i}^\pm\left(C_k^\pm r^{2-N-k}+D_k^\pm r^k\right) Y_{k,i}(\theta)\quad&\text{ for } r\in(R,1],
\end{cases}
\end{equation}
where the constants $B_k^\pm$, $C_k^\pm$ and $D_k^\pm$ are defined as follows
\begin{equation*}
\begin{aligned}
&B_k^- = \frac{1-\sg_c}{\sg_c}R^{-k+1}\left( (N-2+k)R^{2-N-2k} +k\right)/F, \quad 
C_k^- = (\sg_c-1)k R^{-k+1}/F, \quad D_k^-=-C_k^-, \\
&B_k^+ = (N-2+2k)R^{2-N-2k}/F, \quad C_k^+ = (1-\sg_c)k/F,\quad D_k^+ = (N-2+k+k\sg_c)R^{2-N-2k}/F,
\end{aligned}
\end{equation*}
and the common denominator $F= N(N-2+k+k\sg_c)R^{2-N-2k}+k N (1-\sg_c)$.
\end{lemma}
\begin{remark}
The quantity $F=N(N-2+k+k\sg_c)R^{2-N-2k}+k N (1-\sg_c)$ is strictly positive if $R\in(0,1)$ and $k\ge 1$. Indeed we have
\begin{equation*}
F> N(N-2+k+k\sg_c)+k N(1-\sg_c)=N^2+2N(k-1)\ge N^2>0.    
\end{equation*}
\end{remark}
We are now ready to apply the implicit function theorem to the mapping $\Psi$ defined by \eqref{psi is dare}. As a consequence we obtain the following more precise version of Theorem \ref{mainthm1}.
\begin{theorem}\label{ift applied}
Define
\begin{equation*}
\begin{aligned}
s(k)&= \frac{k(N+k-1)-(N+k-2)(k-1)R^{2-N-2k}}{k( N+k-1)+k(k-1)R^{2-N-2k}} \text{ for }k = 1,2,\ldots,\\
\Sigma&=\setbld{s\in (0,\infty)}{s=s(k)
\,\text{ for some }k = 1,2,\ldots}.
\end{aligned}
\end{equation*}
If $\sg_c\notin \Sigma$, then there exists $\varepsilon>0$ such that, for all $f\in \mathcal{X}$ with $\norm{f}_{\mathcal{X}}<\varepsilon$, there exists a unique $g(f)\in\mathcal{Y}$ such that the pair $(D_f,\Omega_{g(f)})$ is a solution of the overdetermined problem \eqref{odp}. Moreover, $\Sigma$ is a finite subset of $(0,1]$.
\end{theorem}
\begin{proof}
This theorem consists in a direct application of the first part of Theorem \ref{ift}. We know that the mapping $(f,g)\mapsto\Psi(f,g)$ is Fr\'echet differentiable and its partial Fr\'echet derivatives were computed in Theorem \ref{Psi'=something}. We now need to prove that the mapping $\pa_y \Psi(0,0): \mathcal{Y}\to \mathcal{Z}$ is a bounded and invertible linear transformation whenever $\sg_c\notin \Sigma$. Since the function $v'_+$ has a linear and continuous dependence on $g$ (see problem \eqref{u'out}), the map defined by $\pa_y \Psi(0,0):\mathcal{Y}\to\mathcal{Z}$ is also linear and bounded. We are left to prove the invertibility of $\pa_y \Psi(0,0)$. To this end, let us write the spherical harmonic expansion of the expression of $\pa_y \Psi(0,0)(g)$ given in \eqref{deri g}, with the aid of Lemma \ref{lemma v' expansion}. Under the assumption \eqref{h_in h_out exp}, we get
\begin{equation}\label{preserves eigenspaces}
\pa_y \Psi(0,0)(g)= \pa_y \Psi(0,0) \Big(\sum_{k=1}^\infty \sum_{i=1}^{d_k} \alpha_{k,i}^+  Y_{k,i}\Big)
= \sum_{k=1}^\infty \sum_{i=1}^{d_k} \beta_k     \alpha_{k,i}^+ Y_{k,i},
\end{equation}
where 
\begin{equation}\label{beta_k}
\beta_k = \frac{(N+k-1)(\sg_c-1)k+(N-2+k+k\sg_c)(k-1)R^{2-N-2k}}{F}.
\end{equation}
In particular, \eqref{preserves eigenspaces} implies that the map $\pa_y \Psi(0,0):\mathcal{Y}\to\mathcal{Z}$ preserves the eigenspaces of the Laplace--Beltrami operator. Moreover, $\pa_y \Psi(0,0)$ is invertible if and only if $\beta_k\ne 0$ for all $k=1,2,\ldots$, that is to say, if and only if $\sg_c\notin\Sigma$.
We will now prove the last assertion of the theorem, namely that $\Sigma$ is a finite subset of $(0,1]$ (and thus, the implicit function theorem can always be applied for $\sg_c>1$). First of all, since $\Sigma\subset (0,\infty)$ by definition, the inclusion $\Sigma\subset (0,1]$ follows from the inequality below (notice that equality holds for $k=1$):
\begin{equation*}
s(k)\le \frac{k(N+k-1)}{k(N+k-1)+k(k-1)R^{2-N-2k}}\le 1.    
\end{equation*}
Moreover, since $s(k)$ tends to $-1$ as $k\to\infty$, the set $\Sigma$ is finite. We remark that the actual cardinality of $\Sigma$ highly depends on the radius $R$. As a matter of fact, $\Sigma$ is empty for small enough values of $R$. On the other hand, as an asymptotic analysis shows, $\Sigma$ can have an arbitrarily large number of elements if we take $R$ sufficiently close to $1$.  
\end{proof}

\begin{remark}\label{about volume constraint}
Notice that the volume constraint $|\Omega_{g(f)}|=|\Omega|$ is {\bf not satisfied} in general (although the discrepancy is $o(\norm{f}_{\mathcal{X}})$ for $\norm{f}_{\mathcal{X}}$ small, as we required elements of $\mathcal{Y}$ to have vanishing integral). Nevertheless, we can apply a small correcting homothety to restore the volume constraint as done in \cite[Corollary 5.6]{cavaphd}. This ensures the existence of a function $\widehat{g}(f)\in \mathcal{C}^{2,\alpha}(\pa\Omega)$ such that the pair $(D_f,\Omega_{\widehat{g}(f)})$ solves the overdetermined problem \eqref{odp} and the volume constraint $|\Omega_{\widehat{g}(f)}|=|\Omega|$.
\end{remark}
\begin{corollary}\label{heredity}
Suppose that $f=\sum_{k=1}^\infty\sum_{i=1}^{d_k}\alpha_{k,i}^- Y_{k,i}$. Then the following first order approximation for $g(f)$ holds true for $f\to 0$:
\begin{equation*}
g(f)=\sum_{k=1}^\infty\sum_{i=1}^{d_k} \frac{\alpha_{k,i}^-(N+2k-2)(\sg_c-1)k R^{1-k}}{(N+k-1)(\sg_c-1)k + (N-2+k+k\sg_c)(k-1)R^{2-N-2k}}  Y_{k,i} + o(\norm{f}_{\mathcal{X}})
\end{equation*}
\end{corollary}
\begin{proof}
This result is a consequence of \eqref{second part} applied to the functional $\Psi$.
We recall that, by \eqref{deri f}, $\pa_x \Psi(0,0)(f)=\pa_n v'_-$. Moreover, by
Lemma \ref{lemma v' expansion}, we get 
\begin{equation*}
\pa_x \Psi(0,0)(f)= \pa_x \Psi(0,0)\Big( \sum_{k=1}^\infty\sum_{i=1}^{d_k} \alpha_{k,i}^- Y_{k,i}   \Big) = \sum_{k=1}^\infty\sum_{i=1}^{d_k} \gamma_k \alpha_{k,i}^-   ,  
\end{equation*}
where 
\begin{equation}\label{gamma_k}
\gamma_k = \frac{2-N-2k}{F}(\sg_c-1)k R^{1-k}.
\end{equation}
By the second part of Theorem \ref{ift}, we get that the map $f\mapsto g(f)$ is Fr\'echet differentiable and 
\begin{equation}\label{g'(f)=...}
g'(f)=g'\Big(  \sum_{k=1}^\infty\sum_{i=1}^{d_k} \alpha_{k,i}^- Y_{k,i} \Big) =  -\sum_{k=1}^\infty\sum_{i=1}^{d_k} \frac{\gamma_k}{\beta_k} \alpha_{k,i}^- Y_{k,i},
\end{equation}
where the coefficients $\beta_k$ and $\gamma_k$ are defined by \eqref{beta_k} and \eqref{gamma_k} respectively. The claim clearly follows from \eqref{g'(f)=...}.
\end{proof}
\section{Shape derivative of the Kohn--Vogelius functional}\label{shape derivative of KV}
In this section we will compute the shape derivative of the Kohn--Vogelius functional $\mathcal{F}$ with respect to perturbations of the outer boundary $\pa\Omega$. To this end, let $h:\rn\to\rn$ be a smooth perturbation field and define $D_t=({\rm Id}+th)(D)$ and $\Omega_t = ({\rm Id}+th)(\Omega)$. Moreover, suppose that $h$ acts only on $\pa\Omega$, that is $D_t\equiv D$ for all $t>0$. Let $v_t= v(D,\Om_t)$, $w_t=w(D,\Omega_t)$ and $c(t)=c(\Omega_t)$. The map $t\mapsto c(t)$ is clearly differentiable at $t=0$ by Lemma \ref{hadamard formula}. Finally, the computations of the shape derivative of the state functions $v_t$ and $w_t$ are contained in the following proposition.
\begin{proposition}
The state functions $v_t$ and $w_t$, defined as the solutions to problems \eqref{pb d} and \eqref{pb n} with $\Omega=\Omega_t$, are shape differentiable, and their shape derivatives $v'$ and $w'$ are characterized as the solutions to the following boundary value problems.

\noindent\begin{minipage}{.35\linewidth}
\begin{equation}\label{v'}
\begin{cases} -\dv(\sg \gr v') =0 \quad\mbox{ in } \Omega,\\
v'= -\partial_n v\, h\cdot n \quad \mbox{ on }\pa\Om.
\end{cases}
\end{equation}
\end{minipage}%
\begin{minipage}{.55\linewidth}
\begin{equation}\label{w'}
\begin{cases}
-\dv(\sg \gr w')=0 \quad \mbox{ in } \Omega,\\
\partial_n w'=-\partial_{nn}w \,h\cdot n + \gr_\tau w\cdot\gr_\tau (h\cdot n) 
+c'
\, \mbox{ on }\partial \Omega,\\
\int_{\partial\Omega} w' = -\int_{\partial\Omega} (c+Hw) \,h\cdot n.
\end{cases}
\end{equation}
\end{minipage}
\end{proposition}
\begin{proof}
The characterization of $v'$ in \eqref{v'} is just \eqref{u'out} with $g=h\cdot n$. 
The derivation of \eqref{w'} is more delicate. First of all, the proof of differentiability of the map $t\mapsto w_t$ is a standard application of Theorem \ref{ift} along the same lines as \cite[Theorem 5.5.1 p.228]{HP2005}. 

We will now prove \eqref{w'}. Let $\varphi\in H^2(\Omega)$ be a given test function. Since $\pa\Omega$ is smooth, it admits an extension to the whole $H^2(\rn)$ (see \cite{sobolev spaces}), which will still be denoted by $\varphi$. We now differentiate the weak form 
\begin{equation*}
\int_{\Omega_t} \sg \gr w_t\cdot \gr \varphi - \int_{\pa\Omega_t} c(t) \varphi = \int_{\Omega_t} \varphi
\end{equation*}
with respect to $t$ at $t=0$, by means of Lemma \ref{hadamard formula}.
We get
\begin{equation*}
\int_{\Omega}\sg \gr w'\cdot \gr \varphi + \int_{\pa\Omega} \gr w\cdot \gr \varphi\, h\cdot n - \int_{\pa\Omega} c' \varphi - \int_{\pa\Omega} \left(cH \varphi + c \pa_n \varphi\right) \, h\cdot n = \int_{\pa\Omega} \varphi\,h\cdot n.
\end{equation*}
By employing the use of the identity 
\begin{equation*}
\gr w\cdot \gr \varphi - c\pa_n\varphi = \gr_\tau w\cdot \gr_\tau \varphi\quad \textrm{ on }\pa\Omega,
\end{equation*}
we get 
\begin{equation}\label{weak form 1}
\int_{\Omega}\sg \gr w'\cdot \gr \varphi + \int_{\pa\Omega} \gr_\tau w\cdot \gr_\tau \varphi \, h\cdot n - \int_{\pa\Omega} c' \varphi - \int_{\pa\Omega} cH \varphi \, h\cdot n = \int_{\pa\Omega}\varphi \, h\cdot n.
\end{equation}
By applying tangential integration by parts \eqref{tangential integration by parts} on the second integral in the above, 
\begin{equation}\label{weak form 2}
\int_{\pa\Omega}\gr_\tau w\cdot \gr_\tau \varphi \, h\cdot n= - \int_{\pa \Omega} \De_\tau w \varphi \, h\cdot n  - \int_{\pa\Omega} \gr_\tau w \cdot \gr_\tau (h\cdot n)\varphi.
\end{equation}
The term $\De_\tau w$ in the above can by handled by combining \eqref{pb n} and the decomposition formula for the Laplace operator into normal and tangential components \eqref{decomp lapl}: 
\begin{equation}\label{decomp lapl 1}
-1=\Delta w = \pa_{nn}w + cH + \De_\tau w \quad \textrm{ on } \pa \Omega.
\end{equation}
By combining \eqref{weak form 1}, \eqref{weak form 2} and \eqref{decomp lapl 1}, we get
\begin{equation*}
\int_{\Omega}\sg \gr w'\cdot \gr \varphi + \int_{\pa\Omega} \left( \pa_{nn}w\,h\cdot n - \gr_\tau w\cdot \gr_\tau (h\cdot n)  -c'     \right) \varphi=0.
\end{equation*}
Since $\varphi$ is arbitrary, this implies the first two lines of \eqref{w'}. Lastly, the normalization condition of \eqref{w'} follows by applying Lemma \ref{hadamard formula} to $\int_{\pa\Omega_t} w_t =0$.
\end{proof}

\begin{theorem}\label{shape kv}
The Kohn--Vogelius type functional $\mathcal{F}$ defined by \eqref{def kv} is shape differentiable at $\Omega$. Moreover, for any smooth $h:\rn\to\rn$ whose support is compactly contained in $\rn\setminus \ol D$, we have  
\begin{equation*}
\mathcal{F}'(\Omega)(h)=\int_{\pa\Omega}\left\lbrace -|\gr w|^2+2(1+cH)
w
-|\gr v|^2+2c^2\right\rbrace h\cdot n.
\end{equation*}
\end{theorem}
\begin{proof}
The differentiability of $\mathcal{F}$ ensues from that of the state functions $v$, $w$, and Lemma \ref{hadamard formula}.
Now, an application of Lemma \ref{hadamard formula} yields
\begin{equation}\label{F'=something}
\begin{aligned}
\mathcal{F}'(\Omega)(h)= 2\int_\Omega \sg (\gr v - \gr w)\cdot (\gr v' - \gr w')+ \int_{\partial\Omega} |\gr v-\gr w|^2 \,h\cdot n\\
= 
2\underbrace{\int_\Omega \sg \gr w \cdot\gr w'}_{(A)} 
- 2\underbrace{\int_\Omega \sg\gr w\cdot \gr v'}_{(B)} +\underbrace{\int_{\partial\Omega} |\gr v-\gr w |^2 \, h\cdot n}_{(C)},
\end{aligned}
\end{equation}
where, in the last equality, we used the fact that $\int_\Omega \sg \gr v\cdot \gr v'=\int_\Omega \sg \gr v \cdot\gr w'=0$ (this ensues by taking $v$ as a test function in \eqref{v'}-\eqref{w'}).
We will now try to write the expression above as the sum of surface integrals on $\partial\Omega$ by using integration by parts. 
Now, taking $w$ as a test function in \eqref{w'} yields
\begin{equation}\label{A=something}
\begin{aligned}
(A)=\int_\Omega \sg \gr w'\cdot \gr w = \int_{\partial\Omega} \partial_n w'\, w = \int_{\partial \Omega} \left(-\partial_{nn} w\, h\cdot n + \gr_\tau w \cdot \gr_\tau (h\cdot n)+c'\right) w.
\end{aligned}
\end{equation}
We now employ once again the use of tangential integration by parts \eqref{tangential integration by parts} to remove the dependence on $\gr_\tau (h\cdot n)$ in the integral above.
We get
\begin{equation}\label{A=something2}
\int_\Omega \left( \gr_\tau w \cdot \gr_\tau(h\cdot n) \right) w= -\int_\Omega {\dv}_\tau \left( w \gr_\tau w\right) \, h\cdot n= -\int_\Omega \left( |\gr_\tau w|^2 +w \De_\tau w \right) \, h\cdot n.
\end{equation}
This can be simplified further. 
By putting together \eqref{decomp lapl 1}, \eqref{A=something}, \eqref{A=something2} and the normalization condition $\int_{\pa\Omega}w=0$, we obtain
\begin{equation}\label{A=something3}
(A)=\int_\Omega \sg \gr w' \cdot \gr w = \int_{\pa\Omega} \left\{ -|\gr_\tau w|^2+ (1 + cH)
w\right\} \, h\cdot n.
\end{equation}
Similarly, by taking $v'$ as a test function in \eqref{pb n} and recalling the boundary condition of \eqref{v'}, we obtain
\begin{equation}\label{B=something}
(B)=\int_\Omega\sg \gr w\cdot\gr v' = \int_{\partial \Omega}\partial_n w \,v' + \int_\Omega v'= -\int_{\pa\Omega} \pa_n w \,\pa_n v h\cdot n +\int_\Omega v'.
\end{equation}
As far as the term $\int_\Omega v'$ is concerned, consider the following integral identity derived from \eqref{pb d}
\begin{equation*}
\int_{\Omega_t} v_t = \int_{\Omega_t} \sg_t |\gr v_t|^2 \quad \textrm{ for } t\ge 0.
\end{equation*}
Differentiating both members of the equality above by means of the Hadamard formula (Lemma \ref{hadamard formula}), yields
\begin{equation}\label{int v'=}
\int_\Omega v' + \underbrace{\int_{\pa \Omega} v\, h\cdot n}_{=0} = \underbrace{2 \int_\Omega \sg \gr v\cdot \gr v'}_{=0} + \int_{\pa\Omega}|\pa_n v|^2\, h\cdot n.
\end{equation}
We can then rewrite \eqref{B=something} as follows:
\begin{equation}\label{B=something2}
(B)=\int_\Omega \sg \gr w\cdot \gr v' = -c\int_{\pa \Omega} \pa_n v\, h\cdot n + \int_{\pa\Omega} |\pa_n v|^2\, h\cdot n 
\end{equation}
Finally, we have
\begin{equation}\label{C=something}
(C)= \int_{\pa\Omega} |\gr v - \gr w|^2\, h\cdot n = \int_{\pa\Omega} \left( |\pa_n v|^2+ c^2 + |\gr_\tau w|^2 - 2c \pa_n v  \right )\,h\cdot n.
\end{equation}
The claim follows by combining \eqref{F'=something} with the final expressions of $(A)$, $(B)$ and $(C)$ in \eqref{A=something3}, \eqref{B=something2} and \eqref{C=something}.
\end{proof}
\begin{remark}
In proving Theorem \ref{shape kv}, we did not make use of the normalization condition in \eqref{w'}. This is natural, since the functional $\mathcal{F}$ depends on $w$ by means of its gradient only. Indeed, for any normalization that we impose on $w_t$, the computations above yield the same result, namely
\begin{equation*}
\mathcal{F}'(\Omega)(h)=\int_{\pa\Omega}\left\lbrace -|\gr w|^2+2(1+cH)
\left(w-\frac{1}{|\pa\Omega|}\int_{\pa\Omega}w\right)
-|\gr v|^2+2c^2\right\rbrace h\cdot n.
\end{equation*}
In light of the expression above, the normalization condition $\int_{\pa\Omega} w=0$, that was chosen in \eqref{pb n}, is indeed the most natural one.
\end{remark}

\section{Gradient descent method with volume constraint}\label{gradient descent method with volume constraint}
In this section, we describe the numerical algorithm for the outer problem (Problem 2). 
This algorithm is based on a gradient descent method using the shape derivative of the Kohn--Vogelius functional \eqref{def kv}, as mentioned above, and coupled with an augmented Lagrangian for the volume constraint.

\subsection{Shape optimization problem and augmented Lagrangian}
Let us recall the shape optimization problem with volume constraint considered in this paper:
\begin{equation}\label{optipro1}
\min_{\abs{\Omega} = V_{0}} \mathcal{F}(\Omega), 
\end{equation}
where $V_{0}$ is a given volume and the minimization is taken over all possible domains
$\Omega$ that satisfy $\ol D\subset \Omega$. We apply the augmented Lagrangian method to the optimization problem \eqref{optipro1} in order to change the problem with volume constraint into a problem without constraints. We refer to \cite[Section 17.3 and Section 17.4]{NW2006} and \cite[Section 3.3]{PT2018} for the details . 

Let us consider the following optimization problem: 
\begin{equation}\label{optipro2}
\min \, \mathcal{L}(\Omega, \ell, b),  
\end{equation}
where $\mathcal{L}$ is the augmented Lagrangian defined by 
\begin{equation}\label{auglag}
\mathcal{L}(\Omega)=\mathcal{L}(\Omega, \ell, b) = \mathcal{F}(\Omega) - \ell G(\Omega) + \frac{b}{2}G(\Omega)^{2},
\end{equation}
and $G(\Omega)$ is the constraint functional given by 
\begin{equation}\label{constraint}
G(\Omega) = \frac{\abs{\Omega} - V_{0}}{V_{0}}. 
\end{equation} 
In the definition of the augmented Lagrangian \eqref{auglag}, the parameter $\ell$ is a Lagrange multiplier associated with the volume constraint \eqref{constraint} and $b$ is a positive parameter for strengthening the volume constraint.  

By Theorem \ref{shape kv} and $\displaystyle G'(\Omega)(h) = \frac{1}{V_{0}}\int_{\pa \Omega} h \cdot n$, we can calculate the shape derivative of the augmented Lagrangian $\mathcal{L}$ as follows: 
\begin{equation}\label{compderi}
\begin{aligned}
\mathcal{L}'(\Omega)(h) &= \mathcal{F}'(\Omega)(h) - \ell G'(\Omega)(h) + b G(\Omega) G'(\Omega)(h) \\
&= \int_{\pa\Omega} \underbrace{\left(-|\gr w|^2+2w+2cHw-|\gr v|^2+2c^2 - \ell + b \, \frac{\abs{\Omega} - V_{0}}{V_{0}^{2}}\right)}_{=:\phi} h \cdot n. 
\end{aligned}
\end{equation}
The computation \eqref{compderi} shows the descent direction for the augmented Lagrangian $\mathcal{L}$. Indeed, if we take the perturbation field $h$ as $h = -\phi n$ on $\pa \Omega$, it follows that for small $t > 0$
\begin{equation*}
\mathcal{L}(\Omega_t) = \mathcal{L}(\Omega) - t \int_{\pa \Omega} \phi^2 + o(t) < \mathcal{L}(\Omega). 
\end{equation*}
Note that the descent direction $\phi$ is defined only on the boundary $\pa \Omega$. From the numerical point of view, it is necessary to extend the descent direction to the whole domain $\Omega$. We choose the popular extension procedure to do this, see \cite{AP2006, Dg2006, PT2018}. 
The basic idea is to introduce a Hilbert space $V$ of regular perturbation fields defined on $\Omega\setminus \ol D$ and then identify the descent direction of $\mathcal{L}$ by representing the shape derivative $\mathcal{L}'(\Omega)(h)$ with respect to a different inner product $(\cdot,\cdot)_{V}$, instead of the usual $(\cdot,\cdot)_{L^2(\pa \Omega)}$. 

For this purpose, the Hilbert space $V$ is defined by 
\begin{equation*}
V=\setbld{h \in H^{1}(\Omega\setminus \ol D, \, \rn)}{h = 0 \,\, \text{on} \,\, \pa D},
\end{equation*}
with inner product 
\begin{equation*}
(h, \xi)_{V} = \gamma \int_{\Omega\setminus \ol D} \nabla h : \nabla \xi + \int_{\Omega\setminus \ol D} h \cdot \xi,   
\end{equation*}
where $\gamma > 0$ is a small parameter and $\gr h:\gr \xi$ is the double contraction defined by ${\rm tr}\left( \gr h (\gr \xi)^T \right)$. We search for $h \in V$ such that for all test function $\xi \in V$, 
\begin{equation}\label{extension1}
(h, \xi)_{V} = -\mathcal{L}'(\Omega)(\xi) = -\int_{\pa \Omega} \phi\, \xi \cdot n. 
\end{equation}
By \eqref{extension1}, we have $\mathcal{L}'(\Omega)(h) = -(h, h)_{V} < 0$. This implies that the solution $h$ of \eqref{extension1} is also a gradient descent direction for $\mathcal{L}$. 
Moreover, by integration by parts, we can regard \eqref{extension1} as the weak form of the following elliptic system: 
\begin{equation}\label{extension2}
\begin{cases}
- \gamma \Delta h + h = 0 \,\, &\text{in} \,\, \Omega\setminus \ol D, \\
h = 0 \,\, &\text{on} \,\, \pa D, \\
\gamma \dfrac{\pa h}{\pa n} = - \phi n \,\, &\text{on} \,\, \pa \Omega. 
\end{cases}
\end{equation}
Note that the regularity of the solution of \eqref{extension2} depends on that of $\phi$. If we assume that $\Omega$ is of class $C^{2,\alpha}$, then by the standard elliptic regularity theory we only obtain $h \in \cC^{1,\alpha}$ on the boundary $\pa \Omega$ because of the regularity of the additive curvature. However, the regularity of $h$ needs to be $\cC^{2,\alpha}$ from the theoretical point of view. The loss of the regularity can cause numerical instability. Thus we have to assume the initial shape $\Omega_{0}$ to be of class $\C^{\infty}$ in the numerical computation. This allows us to have a regularized descent direction $h$ of class $\cC^\infty$ at each iteration. 

\subsection{The algorithm for the numerical computation}\label{secalgori}
In what follows, we describe the algorithm for the numerical computation. We employ the use of the free software FreeFem++ \cite{He2012} which allows us to solve partial differential equations by the finite element method. 

\begin{breakbox}
\begin{itemize}
\item Put initial shape $\Omega_{0}$ and initial values of the Lagrange multiplier $\ell_{0}$ and $b_{0}$ for the augmented Lagrangian. 
\end{itemize}
For $i = 0,1,\cdots$ until convergence: 
\begin{enumerate} 
\item Compute $c(\Omega_{i})=-|\Omega_{i}|/|\pa\Omega_{i}|$. 
\item Solve the Dirichlet problem \eqref{pb d} and the Neumann problem \eqref{pb n}. 
\item Compute $\phi_{i}$ by \eqref{compderi} and the regularized gradient descent $h_{i}$ of the augmented Lagrangian $\mathcal{L}$ by solving \eqref{extension2}. 
\item Take $\varepsilon_{i} > 0$ small enough and move the domain $\Omega_{i}$ according to $h_{i}$: 
\begin{equation*}
\Omega_{i+1} = (\text{Id} + \varepsilon_{i} h_{i}) (\Omega_{i}). 
\end{equation*}
If the mesh reverses, then we take smaller value of $\varepsilon_{i}$. 
\item Update the parameters of the augmented Lagrangian as follows: 
\begin{equation*}
\ell_{i+1} = \ell_{i} - b_{i} \, G(\Omega_{i}) \,\, \text{and} \,\, b_{i+1} = \begin{cases}
\alpha \, b_{i} \quad &\text{if} \,\, b_{i} < b_{\text{max}}, \\
b_{i} \quad &\text{otherwise},  
\end{cases}
\end{equation*}
where $\alpha$ is a small positive parameter and $\alpha > 1$. Also $b_{\rm max}$ is a positive large parameter. 
\end{enumerate}
\end{breakbox}

\begin{remark}
In the algorithm, it is necessary to compute the mean curvature of $\pa \Omega$ when computing $\phi_i$, see \eqref{compderi}. We solved this task by following \cite[pp.430-431]{FG2008} and \cite[Section 3.6]{PT2018}. 
\end{remark}
\begin{remark}\label{remark 5.2}
We note that we have to choose a suitable parameter $\gamma$ to define the regularized extension field $h$ by solving \eqref{extension2}. We took $\gamma = 3.0$ in our computation. Moreover, we also need to pay attention to choose the parameters of the augmented Lagrangian due to the volume constraint. We took the initial parameters of the augmented Lagrangian as $\ell_{0} = \displaystyle \dfrac{V_{0}}{|\pa \Omega_{0}|} \int_{\pa \Omega_{0}} \left(-|\gr w_0|^2+2w_0+2c(\Omega_0)H_0 w_0-|\gr v_0|^2+2c(\Omega_0)^2 \right) $ and $b_{0} = 0.01$, where $v_0$ and $w_0$ are the solutions of the Dirichlet problem \eqref{pb d} and the Neumann problem \eqref{pb n} in the initial shape $\Omega_0$, respectively and $H_0$ is the mean curvature of the the initial shape $\Omega_0$. Furthermore, we took $\alpha = 1.5$ and $b_{\max} = 1000$.
\end{remark}

\section{Numerical results}\label{Numerical results}
In this section, we show the numerical results according to the algorithm presented in section \ref{secalgori}. Unless otherwise specified, we take $\si_c = 10$. 
\subsection{When $D$ is a ball}
In the first example, we show what happens when the core $D$ is a ball. We considered the case where the core $D$ is the disc of radius $2.7$ centered at the origin 
and the initial shape $\Omega_{0}$ is the region enclosed by the curve $\{ ( 0.6\cdot(8+\cos 3t)\cos(t), 0.5\cdot(8+\cos 3t)\sin(t) ) \,\, | \,\, t \in [0,2\pi) \}$ (Figure \ref{fig:1}). 

\begin{figure}[htbp]
\begin{minipage}{0.5\hsize}
\begin{center}
\includegraphics[width=58.2 mm]{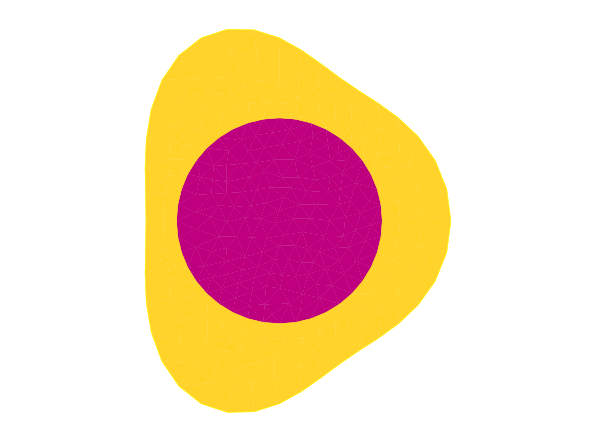}
\end{center}
\end{minipage}
\hfill
\begin{minipage}{0.5\hsize}
\begin{center}
\includegraphics[width=50mm]{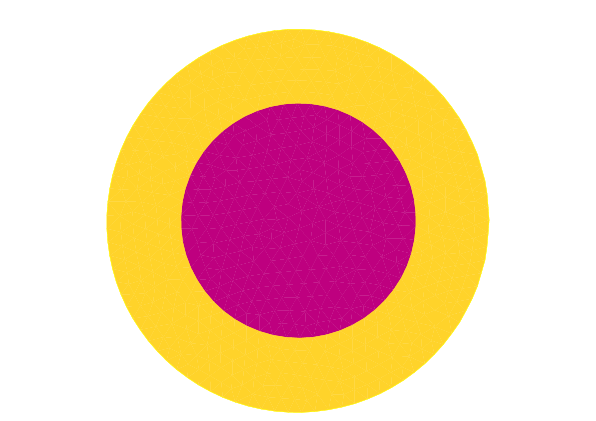}
\end{center}
\end{minipage}
\begin{minipage}{0.5\hsize}
\caption{Initial shape}
\label{fig:1}
\end{minipage}
\hfill
\begin{minipage}{0.5\hsize}
\caption{Final shape}
\label{fig:2}
\end{minipage}
\end{figure}

Figures \ref{fig:1} and \ref{fig:2} show that, if the core $D$ is a ball, then the solution $\Omega$ of outer problem is a ball as suggested by the uniqueness part of Theorem \ref{ift applied} and Remark \ref{about volume constraint}. 
Also, by Figure \ref{fig:3}, we can see that the augmented Lagrangian converges to $0$ oscillating and the volume of $\Omega_n$ also converges to the initial volume in the same way. 
Figure \ref{fig:3} also shows that the augmented Lagrangian is nearly equal to the Kohn--Vogelius functional when the iteration numbers are small. 
\begin{figure}[htbp]
\begin{center}
\includegraphics[width=110mm]{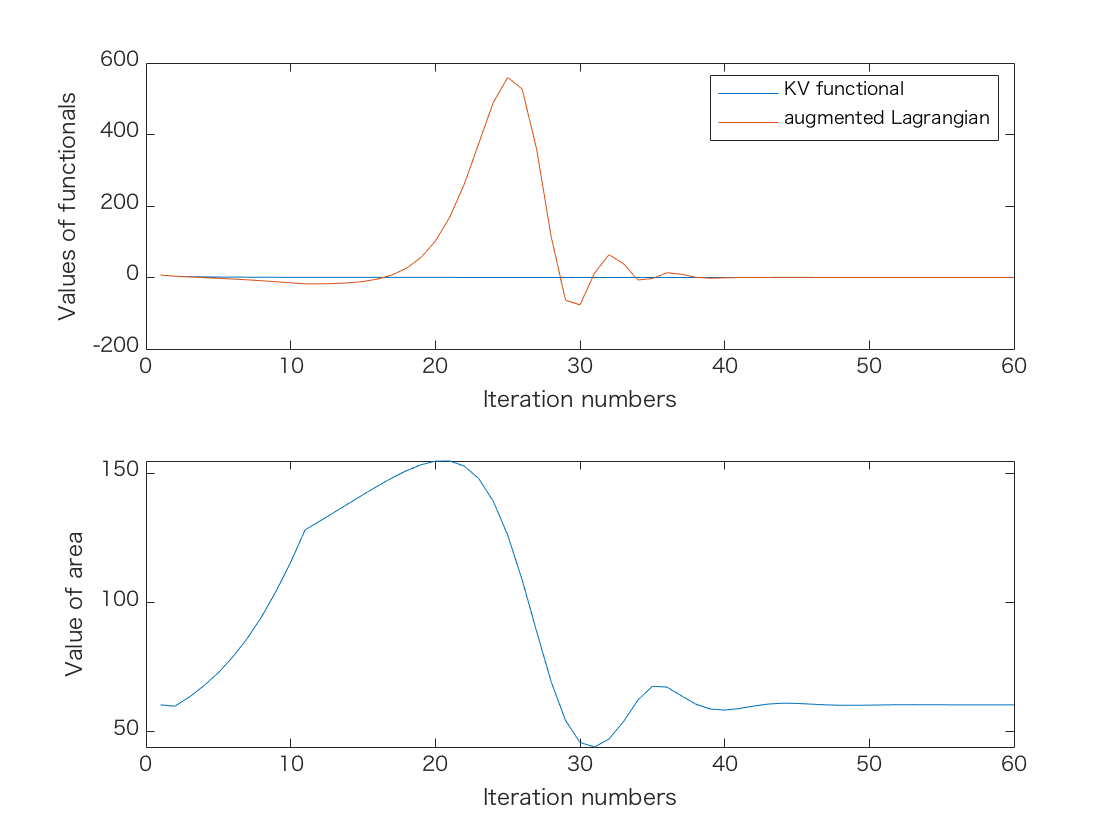}
\end{center}
\caption{Convergence history of functionals and the volume of area}
\label{fig:3}
\end{figure}
By definition, minimizing the augmented Lagrangian $\mathcal{L}$ consists in a compromise between minimizing the Kohn--Vogelius functional $\mathcal{F}$ and the constraint functional $G$. These two tasks are usually in competition with each other: loosely speaking, minimizing the Kohn--Vogelius functional drives $\Omega$ to get larger and more ``rounded" (see also Figures \ref{fig:8} and \ref{fig:9}), which interferes with the minimization of $G$. The balance between these two tendencies is dictated by the choice of the initial parameters $\ell_0$, $b_0$, $\alpha$ and $b_{\rm max}$ and does not remain constant throughout the minimization process.
The initial parameters defined in Remark \ref{remark 5.2} show the following behavior. At first, due to the smallness of the parameters $\ell_{0}$ and $b_{0}$, the Kohn--Vogelius functional drives the minimization process until we get close to a solution of the outer problem with respect to a larger $V_0$. Then $\Omega_i$ starts shrinking in order to fulfill the original volume constraint (depending on the parameters chosen, the algorithm might alternate between the two behaviors described above a few more times in an oscillatory fashion before actually reaching convergence).   
\subsection{$\Omega$ inherits its geometry from $D$}
The second example is in the case where the core $D$ is the region enclosed by the curve $\{ ( 0.3\cdot(8+\cos 3t)\cos t, 0.3\cdot(8+\cos 3t)\sin t ) \,\, | \,\, t \in [0,2\pi) \}$ and the initial shape $\Omega_{0}$ is the interior of the ellipse $\{ ( 4\cos t, 3\sin t ) \,\, | \,\, t \in [0,2\pi) \}$. 
\begin{figure}[htbp]
\begin{minipage}{0.3\hsize}
\begin{center}
\includegraphics[width=57.3 mm]{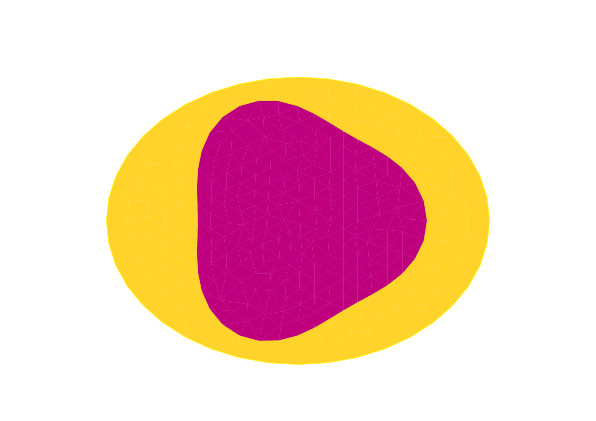}
\end{center}
\end{minipage}
\hfill
\begin{minipage}{0.3\hsize}
\begin{center}
\includegraphics[width=50mm]{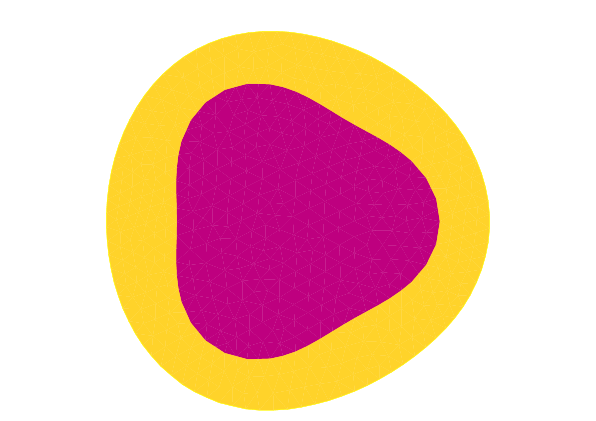}
\end{center}
\end{minipage}
\hfill
\begin{minipage}{0.3\hsize}
\begin{center}
\includegraphics[width=33 mm]{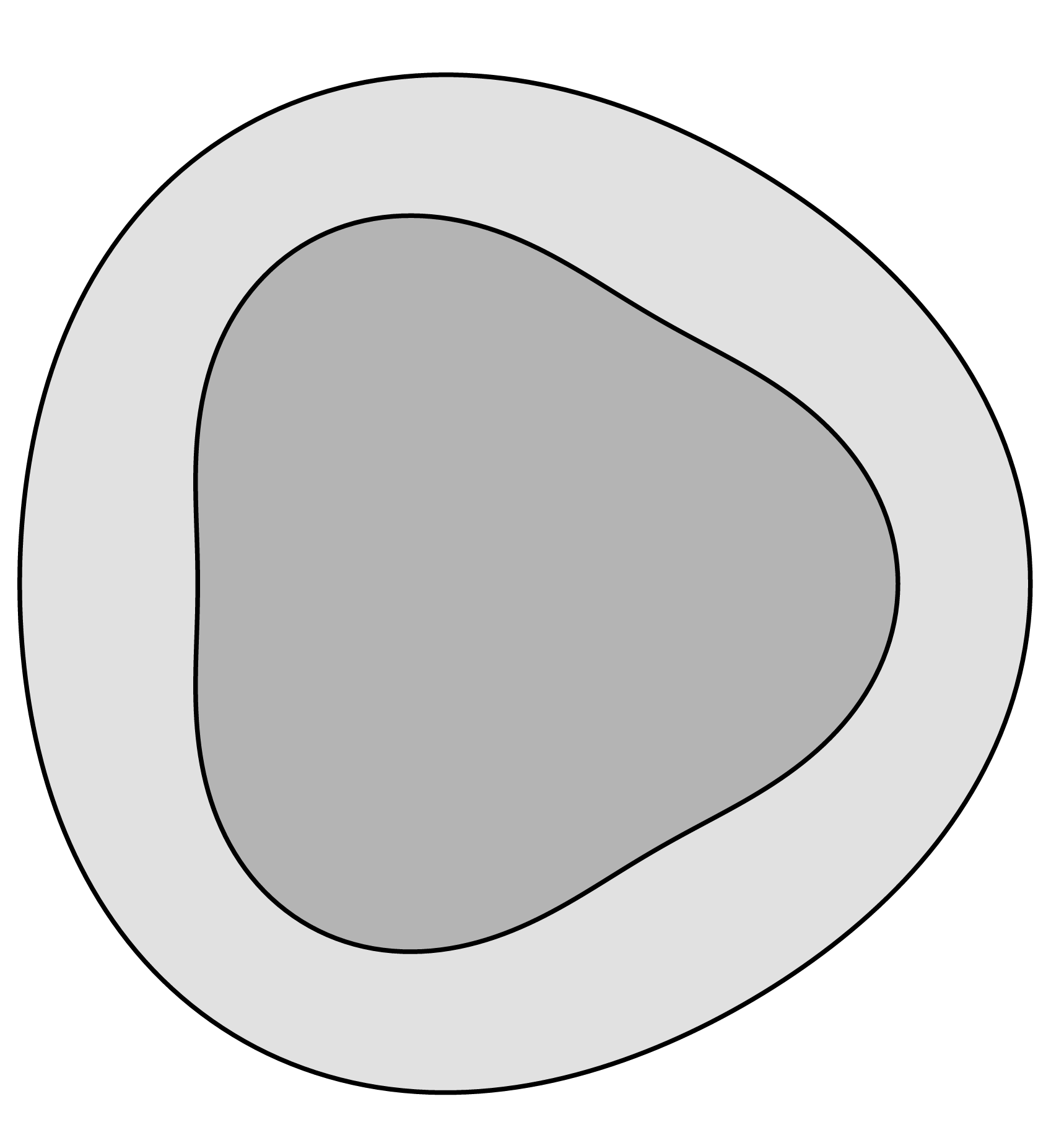}
\end{center}
\end{minipage}
\begin{minipage}{0.3\hsize}
\caption{Initial shape}
\label{fig:5}
\end{minipage}
\hfill
\begin{minipage}{0.3\hsize}
\caption{Final shape}
\label{fig:6}
\end{minipage}
\hfill
\begin{minipage}{0.3\hsize}
\caption{Analytical result}
\label{fig:7}
\end{minipage}
\end{figure}

Figures \ref{fig:5} and \ref{fig:6} show that the solution $\Omega$ of outer problem inherits the geometry of the core $D$ (loosely speaking, the ``bumps" of $\pa D$ and $\pa \Omega$ tend to match). 
Indeed, if $D$ is given by a small perturbation of a ball with normal component defined by the (possibly infinite) sum of some spherical harmonics as in the first expression of \eqref{h_in h_out exp}, Corollary \ref{heredity} shows that the solution $\Omega$ can be approximated by a perturbation of a concentric ball given by a specific weighted sum of the same spherical harmonics.   
Moreover, we can see that the numerical result (Figure \ref{fig:6}) is close to the analytical result given by the first order approximation based on Corollary \ref{heredity}, shown in Figure \ref{fig:7}. 

\subsection{When $D$ is small or $\sigma_c$ is close to $1$}
Here we analyze the two cases where the outer problem can be regarded as a perturbation of the one-phase Serrin problem, namely the case where the core $D$ is small and that where $\sg_c$ is close to $1$.

The third example is in the case where the core $D$ is a sufficiently small domain compared to the initial domain $\Omega_{0}$. We considered the core $D$ as the region enclosed by the curve $\{ ( 0.05\cdot(8+\cos 3t)\cos t, 0.05\cdot(8+\cos 3t)\sin t ) \,\, | \,\, t \in [0,2\pi) \}$. 

The fourth example shows what happens when we take $\sg_c$ sufficiently close to $1$. We defined $D$ to be the region enclosed by the curve $\{ ( 0.3\cdot(8+\cos 5t)\cos t, 0.3\cdot(8+\cos 5t)\sin t ) \,\, | \,\, t \in [0,2\pi) \}$. 
\begin{figure}[htbp]
\begin{minipage}{0.5\hsize}
\begin{center}
\includegraphics[width=50mm]{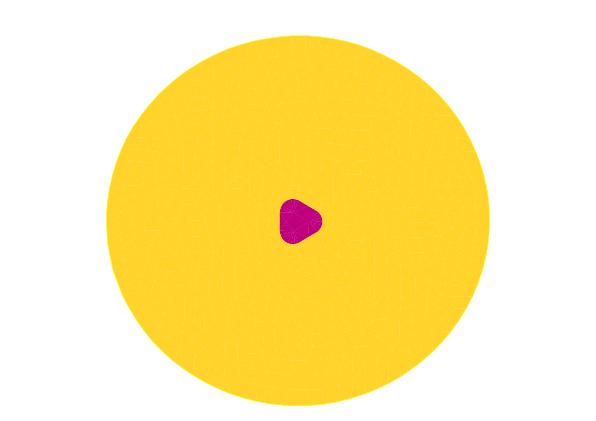}
\end{center}
\end{minipage}
\begin{minipage}{0.5\hsize}
\begin{center}
\includegraphics[width=50mm]{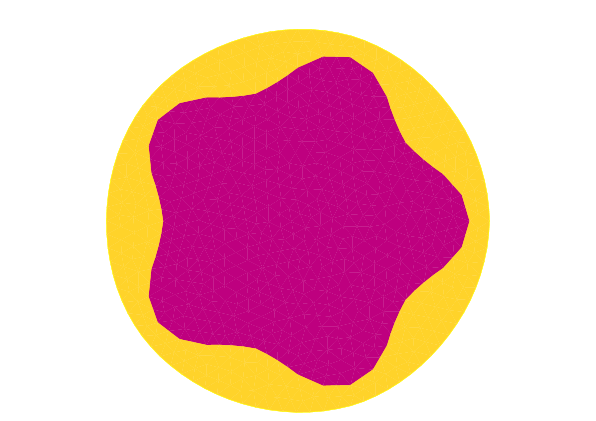}
\end{center}
\end{minipage}
\begin{minipage}{0.5\hsize}
\caption{When $D$ is small}
\label{fig:8}
\end{minipage}
\begin{minipage}{0.5\hsize}
\caption{When $\sg_c$ is close to $1$}
\label{fig:9}
\end{minipage}
\end{figure}

Figures \ref{fig:8} and \ref{fig:9} show that, as one would expect, the solution $\Omega$ of the outer problem is not influenced much by the geometry of the core $D$ and it is nearly a ball if either the size of the core $D$ is sufficiently small or the parameter $\sg_c$ is sufficiently close to $1$. These numerical results justify the intuition that the outer problem (Problem 2) is well approximated by a one-phase Serrin's problem when $D$ is small enough or $\sg_c \simeq 1$. 
\subsection{Different behaviors when $\sg_c\lessgtr 1$}
The fifth example is in the case where the core $D$ is the region enclosed by the curve $\{ ( 0.3\cdot(8+\cos 5t)\cos t, 0.3\cdot(8+\cos 5t)\sin t ) \,\, | \,\, t \in [0,2\pi) \}$ and the initial shape $\Omega_{0}$ is the disk of radius $3$ centered at the origin. Both cases $\sigma_c \lessgtr 1$ are considered. 
\begin{figure}[H]
\begin{minipage}{0.5\hsize}
\begin{center}
\includegraphics[width=50mm]{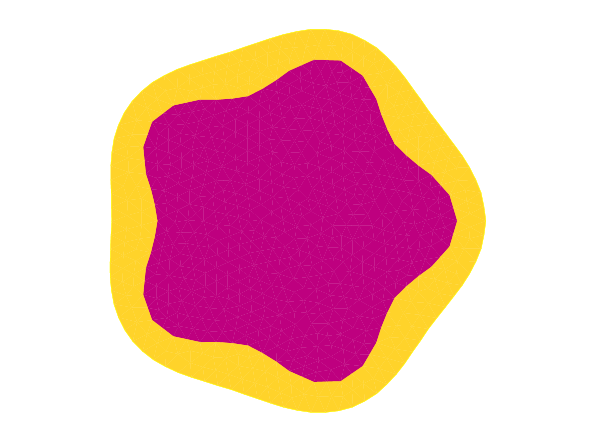}
\end{center}
\caption{Final shape of $\sg_c = 10 > 1$}
\label{fig:11}
\end{minipage}
\begin{minipage}{0.5\hsize}
\begin{center}
\includegraphics[width=50mm]{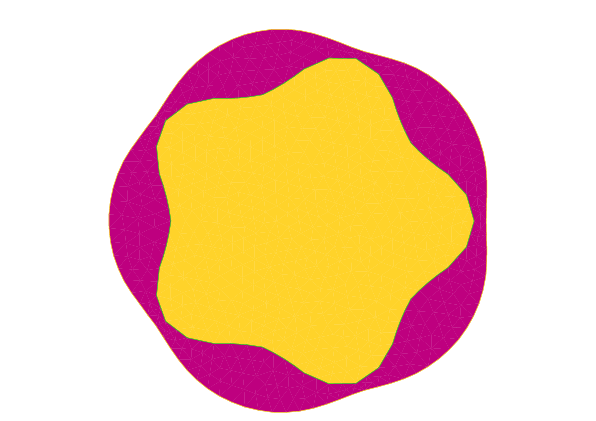}
\end{center}
\caption{Final shape of $\sg_c = 0.1 < 1$}
\label{fig:13}
\end{minipage}
\end{figure}

Figures \ref{fig:11}--\ref{fig:13} show the different behaviors of the cases $\sg_c \lessgtr 1$. If $\sg_c >1$, the solution $\Omega$ presents ``bumps'' that are aligned with those of $D$ in the same direction. On the other hand, if $\sg_c<1$, the ``bumps'' of $\Omega$ point in the opposite direction, thus facing those of $D$. This phenomenon is also predicted by Corollary \ref{heredity}. Indeed, notice that the coefficients $-\gamma_k/\beta_k$, that appear in the Fr\'echet derivative of $g=g(f)$, are positive when $\sg_c>1$ and become negative when $\sg_c<1$ and $k$ is large enough.   

\section{Some open problems and conjectures}\label{pro and con}
In this section, we state some open problems and conjectures. 

\begin{conjecture}\label{conj D small}
For fixed $V_0$, the solution $\Omega$ of the outer problem converges to a ball as the diameter of $D$ tends to $0$ (see Figure \ref{fig:8}).
\end{conjecture}

\begin{conjecture}\label{conj si to 1}
For fixed $V_0$ and $D$, the solution $\Omega$ of the outer problem converges to a ball as $\sg_c\to 1$ (See Figure \ref{fig:9}).
\end{conjecture}
We believe that the techniques developed in \cite{CLM2012, L2014} for the asymptotic expansion of the solution of an eigenvalue problem with respect to the same two-phase operator $-\dv(\sg \gr\cdot )$ might be helpful in proving Conjectures \ref{conj D small}--\ref{conj si to 1}.
Moreover, we notice that Corollary \ref{heredity} constitutes a strong evidence for Conjectures \ref{conj D small}--\ref{conj si to 1}. As a matter of fact, we see that the coefficients $-\gamma_k/\beta_k$ converge to $0$ as $R\to 0$ or $\sg_c\to 1$. Unfortunately, this does not constitute a rigorous proof (not even in the local case). Indeed, Theorem \ref{ift applied} ensures unique solvability of the outer problem with respect to deformed core $D_f$, only for $\norm{f}_{\mathcal{X}}<\varepsilon$, where $\varepsilon$ depends on the parameters $R$ and $\sigma_c$. In particular, we are not allowed to take the limits as $R\to 0$ or $\sg_c\to1$ of the expression in Corollary \ref{heredity} unless we have a uniform estimate on the above-mentioned existence threshold $\varepsilon$. This is a further motivation for studying the following problem.  

\begin{prob}
Study {\bf global} existence and uniqueness for the outer problem.
\end{prob}
 \begin{remark}
Showing global existence and uniqueness is a difficult task at this stage because, to our knowledge, there does not exist any comparison result for this kind of problem. In particular, we did not succeed in generalizing the approach of subsolutions and supersolutions by Beurling (see for example \cite{Beurling 1957, Henrot Shahgholian 2002}). Nevertheless, we think that this might be a valuable tool for proving the following three conjectures. 
\end{remark}

\begin{conjecture}
If $\sigma_c>1$, then there exists a threshold $V^*\ge |D|$ such that for all $V_0>V^*$, the outer problem has a unique solution $\Omega$. In particular, if $D$ is not a ball, then $V^*>|D|$ and the boundaries $\pa D$ and $\pa \Omega$ touch in the limit as $V_0\to V^*$. 
\end{conjecture}

\begin{conjecture}
For fixed $D$, the solutions of the outer problem form an elliptically ordered family. In other words, if $\Omega_1$ and $\Omega_2$ denote two solutions of the outer problem with respect to the same core $D$ and $|\Omega_1|<|\Omega_2|$, then $\Omega_1\subset \Omega_2$.
\end{conjecture}

\begin{conjecture}
If $\sg_c>1$ and $D$ is convex, then any solution $\Omega$ of the outer problem is convex.
\end{conjecture}

\section*{Acknowledgements}
We would like to sincerely thank the anonymous reviewers for their comments and observations, which helped a lot to improve the overall readability of the paper and to better highlight the position of this work in the context of the existing literature.   

\begin{small}

\end{small}
\end{document}